\documentclass[a4paper,11pt]{amsart}
\usepackage{amssymb,mathrsfs}
\usepackage{amsthm}
\usepackage{amsmath}
\usepackage{latexsym}
\usepackage{fancyhdr}
\usepackage{ascmac}
\usepackage{color}
\usepackage[all]{xy}
\usepackage{amscd}

\theoremstyle{plain}
\newtheorem{thm}{Theorem}[section]
\newtheorem{prop}[thm]{Proposition}

\newtheorem{rmk}[thm]{Remark}

\makeatletter

\@addtoreset{equation}{section}
\makeatother

\newcommand{\bQ}{\overline{\mathbb{Q}}}

\newcommand{\C}{\mathbb{C}}
\newcommand{\R}{\mathbb{R}}
\newcommand{\Q}{\mathbb{Q}}

\newcommand{\Z}{\mathbb{Z}}

\newcommand{\F}{\mathbb{F}}

\newcommand{\lra}{\longrightarrow}

\newcommand{\bH}{\mathbb{H}}

\renewcommand{\O}{\mathcal{O}}

\newcommand{\ds}{\displaystyle}

\newcommand{\G}{\Gamma}

\newcommand{\SL}{{\rm SL}}

\newcommand{\e}{\varepsilon}

\newcommand{\bs}{\backslash}

\title[] 
{Computing algebraic Belyi functions on Bring's curve}
\author{Madoka Horie and Takuya Yamauchi}

\thanks{}
\subjclass[2010]{}

\address{Madoka Horie \\
Graduate School of Science and Technology, Sophia University\\
 7-1 Kioi-cho, Chiyoda-ku, Tokyo, 102-8554, JAPAN}
\email{horiemaaa@gmail.com}

\address{Takuya Yamauchi \\ 
Mathematical Inst. Tohoku Univ.\\
 6-3,Aoba, Aramaki, Aoba-Ku, Sendai 980-8578, JAPAN}
\email{takuya.yamauchi.c3@tohoku.ac.jp}

\begin{document}

\maketitle

\begin{abstract}
In this paper, we explicitly compute two kinds of algebraic Belyi functions on Bring's curve. 
One is related to a congruence subgroup of $\SL_2(\Z)$ and the other is related to a 
congruence subgroup of the triangle group $\Delta(2,4,5)\subset \SL_2(\R)$. To carry out the computation, we use elliptic cusp forms of weight 2 for the former case and the automorphism group of Bring's curve for the latter case. We also discuss a suitable base 
field (a number field) for describing isomorphisms between Hulek-Craig's curve, Bring's curve, and another algebraic model obtained as a modular curve. 
\end{abstract}

\section{Introduction}
The main result of Belyi \cite{Belyi} shows that a complex projective smooth curve $C$ has a 
model, as an algebraic curve, defined over a number field if and only if there exists a non-constant function $\beta:C\lra \mathbb{P}^1$ 
ramified at most over $\{0,1,\infty\}$. In this paper, we call such a function 
$\beta$ an algebraic Belyi function and more precisely, an algebraic Belyi function 
defined over a number field $K$ if it is defined over $K$. 
Thus, any projective smooth geometrically connected curve over a number field is endowed with 
an algebraic Belyi function. However, finding such a $\beta$ explicitly for a given $C$ is another matter. In fact, 
the explicit computation of algebraic Belyi functions has been studied extensively
(cf. \cite{SS},\cite{SV} among others). 
Belyi's result can also be understood via hyperbolic uniformization and 
there are two types of uniformization:
\begin{enumerate}
\item (Non-compact case) There exists a finite index subgroup $\G$ of $\SL_2(\Z)$ 
and a biholomorphism $X(\G):=\G\bs (\mathbb{H}\cup\mathbb{P}^1(\Q))\simeq C(\C)$ such that 
$\beta$ is interpreted as the natural projection $$X(\G)\lra X(\SL_2(\Z))\stackrel{\sim}{\lra} \mathbb{P}^1(\C)
,\ \G\tau\mapsto \SL_2(\Z)\tau\mapsto \frac{1}{1728}j(\tau)$$  
where $j$ is Klein's $j$-invariant function, with expansion $j(\tau)=q^{-1}+744+196884q+\cdots,\ 
q:=e^{2\pi \sqrt{-1}\tau},\ \tau\in\mathbb{H}:=\{\tau\in\C\ |\ {\rm Im}(\tau)>0\}$. 
\item (Compact case) There exists a finite index subgroup $\Delta$ of 
some triangle group $\Delta(a,b,c)\subset \SL_2(\R)$  
and a biholomorphism $X_\Delta:=\Delta\bs\mathbb{H}\simeq C(\C)$ such that 
$\beta$ is interpreted as the natural projection $$X_\Delta\lra X_{\Delta(a,b,c)}
\stackrel{\sim}{\lra} \mathbb{P}^1(\C)
,\ \Delta\tau\mapsto \Delta(a,b,c)\tau\mapsto j_{\Delta(a,b,c)}(\tau)$$
where $j_{\Delta(a,b,c)}$ is a uniformizer of $X_{\Delta(a,b,c)}$. 
\end{enumerate}
We refer to \cite[p.71, Theorem]{Se97} for the non-compact case and 
\cite[Section 3.3, Theorem 3]{Wolfart} for the compact case.

In this paper, using the above two interpretations, we compute two algebraic Belyi functions on Bring's curve defined by
\begin{equation}\label{Bring}
B:\left\{
\begin{array}{l}
x_0+x_1+x_2+x_3+x_4=0 \\
x^2_0+x^2_1+x^2_2+x^2_3+x^2_4=0 \\
x^3_0+x^3_1+x^3_2+x^3_3+x^3_4=0 
\end{array}\right.
\end{equation}
inside $\mathbb{P}^4$ with homogeneous coordinates $[x_0:x_1:x_2:x_3:x_4]$. 
The symmetric group $S_5$ of degree 5 acts on $B$ by permutation of coordinates and it is known that $S_5$ is 
the full automorphism group of $B$. We refer to \cite[Example 5.5 and Proposition 5.6]{Z} for another aspect of Bring's curve. 
  
A key ingredient in computing the Belyi maps is to relate Bring's curve with a certain modular curve and Hulek-Craig's curve. We briefly introduce 
these curves and the details are given in the corresponding sections.  
Let $X(\G_B)^{{\rm alg}}$ be the algebraic model over $\Q$ of the compactified modular curve $X(\G_B):=\overline{\G_B\bs \mathbb{H}}$ 
which is given by 
\begin{equation}\label{XBalg}
X(\G_B)^{{\rm alg}}:\left\{
\begin{array}{l}
-(x^2+y^2)+2(z^2+w^2)=0 \\
x^3+xy^2+2y^3-4xz^2-8xzw=0
\end{array}\right.
\end{equation}
inside $\mathbb{P}^4$ (see Section \ref{mc}). Here $\G_B:=\Gamma_0(2)\cap \G(5)$ is a congruence subgroup in $\SL_2(\Z)$. 
Let $\beta^{{\rm alg}}_{\G_B}:X(\G_B)^{{\rm alg}}\lra \mathbb{P}^1_\Q$ be an algebraic model over $\Q$ of 
$X(\G_B)\lra X(\SL_2(\Z))\simeq \mathbb{P}^1(\C), \tau\mapsto \frac{1}{1728}j$ 
(see Theorem \ref{belyi1}). 
 
On the other hand, Hulek-Craig's curve $X^{{\rm HC}}$ over $\Q$ is defined by the desingularization of a projective geometrically irreducible singular curve $$F(x_1,y_1,z_1):=x_1(y^5_1+z^5_1)+x^2_1y^2_1z^2_1-x^4_1y_1z_1-2y_1^3z^3_1=0$$
inside $\mathbb{P}^2$ with the coordinates $[x_1:y_1:z_1]$. This curve has studied in \cite{H} and \cite{C}. 
Hulek-Craig's curve $X^{{\rm HC}}$ appears in the context of the Horrocks-Mumford
bundle \cite{H}. Through its symmetries, the latter is also closely related
to the group $\G(5)$.

Then we prove the following result:
\begin{thm}\label{submain}{\rm (}Theorem \ref{isom}{\rm)} Keep the notation as above. Let $\zeta_5$ be a primitive fifth root of unity.  
\begin{enumerate}
\item Bring's curve $B$ is isomorphic over $\Q(\zeta_5)$ to Hulek-Craig's curve $X^{{\rm HC}}$;
\item  Hulek-Craig's curve $X^{{\rm HC}}$ is isomorphic over $\Q$ to $X(\G_B)^{{\rm alg}}$. 
\end{enumerate}
\end{thm}
We note that the first part of Theorem \ref{submain} is due to \cite[Proposition 2.9]{BDH}.
The main results of this paper are Theorem \ref{belyi1} and Theorem \ref{belyi2} and they are summarized as follows: 
\begin{thm}\label{main} Bring curve has two algebraic Belyi functions which are given by the following contexts:  
\begin{enumerate}
\item {\rm (}non-compact type{\rm) }The composition of $B\stackrel{\Q(\zeta_5)}{\simeq}X^{{\rm HC}}
\stackrel{\Q}{\simeq} X(\G_B)^{{\rm alg}}$ and $\beta^{{\rm alg}}_{\G_B}:X(\G_B)^{{\rm alg}}\lra \mathbb{P}^1_\Q$ gives an algebraic Belyi map over $\Q(\zeta_5)$ on $B$. 
\item  {\rm (}compact type{\rm) } The natural quotient $B\lra B/S_5$ gives an algebraic Belyi map over $\Q$ on $B$. 
It is an algebraic model over $\Q$ of the natural quotient map $X_\Delta\lra X_{\Delta(3,4,5)}\simeq \mathbb{P}^1(\C)$ induced by the inclusion of a certain subgroup $\Delta$ of $\Delta(3,4,5)$. 
\end{enumerate}

\end{thm}

Let us explain about some key points for the main results. 
It is known that 
\begin{enumerate}
\item (non-compact case  \cite[p.82-83]{H}) $B(\C)\simeq X(\G_B)$ as a Riemann surface, and  
\item (compact case  \cite[p.73]{CV}) $B(\C)\simeq \Delta\bs \bH$ for some normal congruence subgroup $\Delta$ of $\Delta(3,4,5)$ such that 
$\Delta(3,4,5)/\Delta \simeq  S_5$. Thus,  
${\rm Aut}(B(\C))\simeq {\rm Aut}(\Delta\bs \bH)= \Delta(3,4,5)/\Delta \simeq  S_5$.  
\end{enumerate}

For the non-compact case, we compute an algebraic model of the $j$-function in terms of the modular functions 
on $X(\G_B)$ by regarding $j$ and modular functions with algebraic coordinates. 

For the compact case, we first use the uniformization to conclude that the natural quotient map 
$B(\C)\mapsto B(\C)/{\rm Aut}_{\C}(B(\C))$ is an algebraic Belyi function. 
As a projective smooth complex curve, Bring's curve is the unique curve of genus 4 with $S_5$ as the automorphism group. On the other hands, $S_5$ acts properly on the coordinates of $B$ and 
if we view $B$ as an algebraic curve over $\Q$, we have ${\rm Aut}_\Q(B)\supset S_5$ where 
${\rm Aut}_\Q$ stands for the group of all automorphims defined over $\Q$. However, 
$S_5\subset {\rm Aut}_\Q(B)\subset  {\rm Aut}_\C(B(\C))\simeq S_5$. Thus, ${\rm Aut}_\Q(B)=S_5$. 
Hence, the natural quotient map decent to the quotient map 
$B\lra B/{\rm Aut}_\Q(B)\simeq \mathbb{P}^1_\Q$ which is an algebraic Belyi function defined over $\Q$. 

We will organize this paper as follows. 
In Section \ref{mc}, we study another model of $B$ which is explicitly given as an algebraic model of  the modular curve $X(\G_B)$. 
By using elliptic cusp forms of weight 2 and the theory of canonical curves, we compute 
an explicit defining equation over $\Q$ for $X(\G_B)$. 
In Section \ref{tc}, we will compute the quotient map 
$B\mapsto B/{\rm Aut}_{\Q}(B)$ explicitly. 
Finally, in Section \ref{tab},  we relate our algebraic model of $X(\G_B)$ with Hulek-Craig's curve 
and discuss a number field $K$ such that these curves are 
isomorphic over $K$ to Bring's curve. 
This shows that Bring's curve inherits two interesting algebraic Belyi functions. 

\textbf{Acknowledgments.} We would like to thank Professors K.~Hulek and H.~Braden 
for helpful discussions. We also thank Natalia Amburg for kindly informing us an article \cite{Z}. 
Finally, we would like to give special thanks to the referee, whose suggestions have greatly improved the presentation and readability of this paper.

\section{Bring's curve as a modular curve}\label{mc}
In this section, we compute the Belyi function on Bring's curve $B$ as 
the modular curve $X(\G_B)$ defined below. 
It will be explained in Section \ref{tab} that 
$B$ is isomorphic over $\Q(\zeta_5)$ to an algebraic model over $\Q$ of $X(\G_B)$. Thus, we may work on $X(\G_B)$ to compute the Belyi function. 

We refer to \cite{DS} for the basic facts of elliptic modular forms and modular curves. 
For each positive integer $N$, we denote by $\G(N)\supset \G_0(N)\supset \G_1(N)$ the three kinds of congruence subgroups inside $\SL_2(\Z)$ introduced in \cite[p.13, Definition 1.2.1]{DS}. For such a congruence subgroup $\G$, we define the 
(compact) modular curve $X(\G):=\G\bs (\mathbb{H}\cup\mathbb{P}^1(\Q))$ 
and put 
$X_0(N)=X(\G_0(N))$
 for simplicity.  
We denote by $S_2(\G)$ the space of all cusp forms of weight 2 with respect to $\G$. 
It is known that we have an isomorphism $S_2(\G)\stackrel{\sim}{\lra}H^0(X(\G),\Omega^1_{X(\G)}),\ f(\tau)\mapsto f(\tau)d\tau$ where the right hand side stands for the space of 
all holomorphic 1-forms on $X(\G)$. 
\subsection{Defining equations over $\Q$ of $X(\G_B)$}
Let 
$$\G_B=\Bigg\{
\left(
\begin{array}{cc}
a & b \\
c & d
\end{array}\right)\in \SL_2(\Z)\ \Bigg|\ a\equiv d\equiv 1\ {\rm mod}\ 5,\ c\equiv 0\ {\rm mod}\ 50
\Bigg\}
\subset \G_1(5)\cap \G_0(50).$$
As explained in \cite[Section 2]{BDH}, $B(\C)$ is isomorphic to $X(\G_B)$ as a compact Riemann surface. 
Notice that $\left(
\begin{array}{cc}
5^{-1} & 0 \\
0 & 1
\end{array}\right)\G_B\left(
\begin{array}{cc}
5 & 0 \\
0 & 1
\end{array}\right)=\Gamma_0(2,5):=\G_0(2)\cap \G(5)$ in the notation there 
and also $\G_1(50)\subset \G_B\subset \G_0(50)$.  
Then, the natural inclusion $\G_B\subset \SL_2(\Z)$ yields the quotient map 
\begin{equation}\label{beta1}
\beta_{\G_B}:X(\G_B)\lra X(\SL_2(\Z))\stackrel{\sim}{\lra} \mathbb{P}^1(\C),\ \G_B\tau\mapsto
\SL_2(\Z)\tau\mapsto \frac{1}{1728}j(\tau).
\end{equation}
It is well-known (cf. \cite[Proposition 2.1]{Filom}) that $\beta_{\G_B}$ is a Belyi function ramified only at 
$0,1,\infty$. 
Since each of $X(\G_B)$ and $X(\SL_2(\Z))$ has a structure as a projective smooth algebraic curve defined over $\Q$, the Belyi function $\beta_{\G_B}$ can be seen as an algebraic function 
$\beta^{{\rm alg}}_{\G_B}$ over 
$\Q$. Henceforth, we will explicitly compute $\beta^{{\rm alg}}_{\G_B}$.   

Let us first compute defining equations over $\Q$ of $X(\G_B)$. 
\begin{prop}The modular curve  $X(\G_B)$ has the following  algebraic model over $\Q$ as 
a projective smooth geometrically connected algebraic curve defined over $\Q:$
\begin{equation}\label{X}
X(\G_B)^{{\rm alg}}:\left\{
\begin{array}{l}
-(x^2+y^2)+2(z^2+w^2)=0 \\
x^3+xy^2+2y^3-4xz^2-8xzw=0
\end{array}\right.
\end{equation}
inside $\mathbb{P}^3$ with homogeneous coordinates $[x:y:z:w]$.
\end{prop}
\begin{proof}Under the isomorphism $(\Z/50\Z)^\times\simeq \G_0(50)/\G_1(50)$, 
the subgroup $\G_B/\G_1(50)\subset \G_0(50)/\G_1(50)$ corresponds to 
$\langle \overline{11}:=11\ {\rm mod}\ 50\rangle\simeq \Z/5\Z$. Thus, we have $\G_0(50)/\G_B\simeq 
\Z/4\Z$. It follows from this and \cite{LMFDB} that 
$$S_2(\G_B)=\bigoplus_{\chi\in \widehat{(\Z/50\Z)^\times}\atop \chi^4=1}S_2(\G_0(50),\chi)=
S_2(\G_0(50))\oplus S_2(\G_0(50),\chi_2)$$
where $\chi_2:(\Z/50\Z)^\times=\langle \overline{3}:=3\ {\rm mod}\ 50 \rangle \lra 
\C^\times,\ \overline{3}^i\mapsto (-1)^i$.
We refer to \cite[Section 5.2, p.169]{DS} for the first equality. 
As for generators,  we have $S_2(\G_0(50))=\langle f_1,f_2\rangle$ and $S_2(\G_0(50),\chi_2)=\langle f_3,f_4\rangle$ 
such that their $q$-expansions are given by 
\begin{equation}\label{qexp}
\begin{array}{l}
f_1(\tau):=f_{{\rm 50.2.a.a}}(q)=q - q^2 + q^3 + q^4 - q^6 + 2q^7 - \cdots  \\
f_2(\tau):=f_{{\rm 50.2.a.b}}(q)=q + q^2 - q^3 + q^4 - q^6 - 2q^7 + q^8 -\cdots \\ 
f_3(\tau):=\frac{1}{2}(f_{{\rm 50.2.b.a}}(q,\sqrt{-1})+f_{{\rm 50.2.b.a}}(q,-\sqrt{-1}))=q - q^4 - q^6 +\cdots \\ 
f_4(\tau):=\frac{1}{2\sqrt{-1}}(f_{{\rm 50.2.b.a}}(q,\sqrt{-1})-f_{{\rm 50.2.b.a}}(q,-\sqrt{-1}))= q^2 + q^3 - 2 q^7 - q^8 -
\cdots
\end{array}
\end{equation}
where each label $\ast$ of $f_\ast$ corresponds to the corresponding datum on \cite{LMFDB} 
and we can read off the Fourier coefficients up to $q^{99}$ for which it is enough for the computation below.  
Readers may visit the website https://www.lmfdb.org/ModularForm/GL2/Q/holomorphic/ linked in \cite{LMFDB} and 
enter 50 for the level, and 2 for the weight, and either 1 or 2 for the character order. 
Then all labelled forms as above will show up individually.  

Since Bring's curve is non-hyperelliptic and of genus 4, its canonical divisor is very ample.
Hence, by noting $S_2(\G_B)\stackrel{\sim}{\lra}H^0(X(\G_B),\Omega^1_{X(\G_B)}),\ f(\tau)\mapsto f(\tau)d\tau$, we have the canonical embedding:
$$\Phi:X(\G_B)\hookrightarrow \mathbb{P}^3(\C),\ \tau\mapsto 
[x:y:z:w],\ (x,y,z,w):=(f_1(\tau),f_2(\tau),f_3(\tau),f_4(\tau)).$$
By using Petri's theorem as explained in \cite[Section 2.2]{SM} for $X_0(N)$ 
(which also applies to $\G_B$), the image of $\Phi$ is given by the intersection of 
a quadratic equation and a cubic equation in $\mathbb{P}^3$. The claim follows by determining 
the coefficients of those equations by using (\ref{qexp}); see \cite{SM} for computational details.
\end{proof}

\subsection{An algebraic Belyi function on $X(\G_B)^{{\rm alg}}$}
In this subsection, we compute the algebraic function over $\Q$ corresponding to the quotient map $X(\G_B)\lra  
X(\SL_2(\Z))\stackrel{\SL_2(\Z)\tau\mapsto \frac{1}{1728}j(\tau)}{\simeq}  \mathbb{P}^1(\C)$ by using the intermediate modular curve $X_0(50)$ so that the quotient map factors through it:
$$X(\G_B)\stackrel{\pi_1}{\lra}  X_0(50)\stackrel{\pi_2}{\lra} X(\SL_2(\Z))
\stackrel{j\atop \sim}{\lra} \mathbb{P}^1(\C).$$
Here $\pi_1$ and $\pi_2$ are natural projections induced by the inclusion 
$\G_B\subset \G_0(50)\subset \SL_2(\Z)$. 
Let $X_0(50)^{{\rm alg}}$ be the hyperelliptic curve over $\Q$ defined by 
\begin{equation}\label{m50}
t^2=s^6-4s^5-10s^3-4s+1
\end{equation}
such that $X_0(50)^{{\rm alg}}(\C)\simeq X_0(50)$ as a compact Riemann surface (see \cite[p.290]{HM}). 
An algebraic function over $\Q$ for $j\circ \pi_2$ has been computed in \cite[p.290]{HM} 
with (\ref{m50}) is explicitly given by 
$$X_0(50)^{{\rm alg}}_\Q\lra \mathbb{P}^1_\Q,\ (s,t)\mapsto j=h(s,t):=\frac{A(s)+B(s)t}{2s^{25}(s^4-s^3+s^2-s+1)^2}$$
where the polynomials $A(x)$ and $B(x)$ in the one variable $x$ are given in \cite[p.290]{HM}. 
Thus, we have only to compute an algebraic map between $X$ and (\ref{m50}) 
corresponding to $\pi_1$.  
 
Put $T:=y/z=f_1/f_2$ and $S:=z/w=f_3/f_4$ so that $T,S\in \C(X_0(50))$. By using Mathematica version 12.1, we can compute 
a Gr\"oebner basis for the ideal generated by $zT-y,wS-z$ and (\ref{X}) in $\Q[x,y,z,w,T,S]$ such that  
we can find a relation $1 + S^2 + 2 S^3 - 4 T - 4 S^2 T - T^2 - S^2 T^2 + 2 S^3 T^2=0$ in that basis. The curve defined by this equation is isomorphic over $\Q$ to (\ref{m50}) by  
$$(S,T)=\Bigg(\frac{s+1}{s-1},\frac{s^3-s^2+s-1+t}{s(s^2+s+2)}\Bigg),\ 
(s,t)=\Bigg(\frac{S+1}{S-1}, \frac{2 (-2 - 2 S^2 + T + S^2 T + 2 S^3 T)}{(S-1)^3}\Bigg).$$
Thus, we have  proved the following theorem:
\begin{thm}\label{belyi1}
The algebraic function $$\beta^{{\rm alg}}_{\G_B}:X(\G_B)^{{\rm alg}}\lra \mathbb{P}^1_\Q,\ [x:y:z:w]\mapsto 
\frac{1}{1728}h\Bigg(\frac{z+w}{z-w},\frac{2 (w^3 x - 2 w^3 y + w x z^2 - 2 w y z^2 + 2 x z^3)}{y (z - w)^3}  \Bigg)$$ is an algebraic Belyi function defined over $\Q$. It is also an algebraic model over $\Q$ of $X(\G_B)\lra X(\SL_2(\Z))\simeq \mathbb{P}^1(\C),\tau\mapsto \frac{1}{1728}j(\tau)$.
\end{thm}

\begin{rmk} The Belyi function $X(\G_B)\lra X(\SL_2(\Z))\simeq \mathbb{P}^1(\C)$ 
in question is a non-Galois covering of $\mathbb{P}^1(\C)$. Thus, it gives rise to a non-regular dessin on $X(\G_B)$. 
\end{rmk}

\subsection{A decomposition of the Jacobian variety over a number field}\label{decompJ}
In this section, we study a number field $L$ such that the Jacobian variety ${\rm Jac}(X(\G_B)^{{\rm alg}})$ of $X(\G_B)^{{\rm alg}}$ decomposes completely over $L$.
For two abelian varieties $A_1,A_2$ over a number field $K$ and a field extension $M/K$, 
we write $A_1\stackrel{M}{\sim}A_2$ if $A_1$ is isogenous over $M$ to $A_2$. 
For each newform $f$ in $S_2(\G_B)$, we denote by $A_f$ the Shimura's abelian variety over $\Q$ 
of dimension $[\Q_f:\Q]$ where $\Q_f$ is the Hecke field of $f$ (cf. \cite[Section 3]{Ribet}). We note that $A_f$ is $\Q$-simple.  
\begin{prop}\label{decomp} Keep the notation being in Proposition \ref{X}. 
Let $L=\Q(\sqrt{5},\sqrt{5+2\sqrt{5}})$.
Then, it holds that
$${\rm Jac}(X(\G_B)^{{\rm alg}})\stackrel{\Q}{\sim}A_{f_{{\rm 50.2.b.a}}}\times J_0(50)_\Q
\stackrel{\Q}{\sim}A_{f_{{\rm 50.2.b.a}}}\times E_{{\rm 50.a3}}\times E_{{\rm 50.b3}}\stackrel{L}{\sim}
E^4$$
where $E_\ast$ is the elliptic curve over $\Q$ with conductor 50 whose label $\ast$ corresponds to one in \cite{LMFDB} and $E=E_{{\rm 50.a3}}:y^2+xy+y=x^3-x-2$ is an elliptic curve over $\Q$ with $j$-invariant $-\frac{25}{2}$.
\end{prop}
\begin{proof}We refer to 
$$\text{https://www.lmfdb.org/EllipticCurve/2.2.5.1/100.1/a/3}$$
linked from 
 $$\text{https://www.lmfdb.org/ModularForm/GL2/Q/holomorphic/50/2/b/a/}$$
for \cite{LMFDB}. 
The decomposition over $\Q$ follows from \cite[Proposition (2.3)]{Ribet80}. 
It follows from \cite{LMFDB} (see the above websites)  that  
$A_{f_{{\rm 50.2.b.a}}}\stackrel{\Q}{\sim}
{\rm Res}_{\Q(\sqrt{5})/\Q}(E_1)$ where 
$E_1=E_{100.1-{\rm a}3}:y^2+(\e+1)xy+(\e+1)y=x^3+\e x^2,\ \e:=\ds\frac{1+\sqrt{5}}{2}$ with the $j$-invariant $-\frac{25}{2}$. Here ${\rm Res}_{\Q(\sqrt{5})/\Q}$ stands for the Weil-restriction 
from $\Q(\sqrt{5})$ to $\Q$.
Since $E_1$ and $E$ are isomorphic each other over $\bQ$, by \cite[Chapter III, Proposition 3.1]{Sil}, 
there exist $u,r,s,t\in \bQ$ such that $E\stackrel{\sim}{\lra}E_1,\ (x,y)\mapsto 
(u^2x+r,u^3y+u^2 sx+r)$. Solving the system of the algebraic equations, we have
$$s^2+s+\e^2=0,\ r=-\e,\ u=\e^{-2}(1+2s).$$
On the other hand, by  \cite[Chapter III, Proposition 3.1]{Sil} again, $E=E_{{\rm 50.a3}}\stackrel{\Q(\sqrt{5})}{\simeq} E_{{\rm 50.b3}},\ (x,y)\mapsto (u'^2x+r',u'^3y+u'^2 s'x+r')$ with 
$r'=0,\ s' =\ds\frac{1}{10} (-5 + \sqrt{5}),\ t'= \frac{1}{50} (-25 + \sqrt{5}),\ u' = \frac{\sqrt{5}}{5}$. 
Thus, the claim follows. 
\end{proof}

\begin{rmk}\label{decomp-remark} Since $A_{f_{{\rm 50.2.b.a}}}$ is a $\Q$-simple abelian surface, ${\rm Jac}(X)$ 
never be isogenous over $\Q$ to the product of four elliptic curves over $\Q$. 
\end{rmk}

\section{Bring's curve as a triangular (Shimura) curve}\label{tc}
Let $\Delta(2,4,5)\subset \SL_2(\R)$ be the triangle group corresponding to the 
triple $(a,b,c)=(2,4,5)$  (see \cite[Section 2]{CV}) so that 
$F=\Q(\sqrt{2},\sqrt{5})$ and $E=\Q(\sqrt{5})$ in the notation there. 
By \cite[Proposition 5.13]{CV}, there exists an embedding 
$\Delta(2,4,5)\hookrightarrow \SL_2(\O_E)$ where $\O_E$ stands for the ring of inters of $E$. We define $\Delta$ by 
the kernel of $\Delta(2,4,5)\lra \SL_2(\O_E)\lra \SL_2(\O_E/(\sqrt{5}))$ and it yields the 
isomorphism ${\rm Aut}(X_{\Delta})=\Delta(2,4,5)/\Delta\simeq \SL_2(\O_E/(\sqrt{5}))\simeq \SL_2(\F_5)\simeq S_5$. 
To be more precise, $\Delta(2,4,5)/\{\pm 1\}$ can be regarded as the group of the reduced norm 1 
of a maximal order inside the quaternionic algebra $A:=\Big(\frac{-1,\e}{E}\Big)$ over $E$ 
where $\e=\frac{1+\sqrt{5}}{2}$ 
(see \cite[p.53]{CV}). Thus, $X_{\Delta(2,4,5)}=(\Delta(2,4,5)/\{\pm 1\})\bs \bH$ is a Shimura curve 
(of genus zero).  
Therefore, we have a $S_5$-Galois covering $X_\Delta\lra X_{\Delta(2,4,5)}$ between 
triangular Shimura curves which is a Belyi function by construction.  
The quotient map $B\lra B/{\rm Aut}_\Q(B)=B/S_5\simeq \mathbb{P}^1_\Q$ is an algebraic 
Belyi function over $\Q$ since its base change to $\C$ is corresponding to the quotient map $X_{\Delta(2,4,5)}\lra X_\Delta$. 
Thus, we have only to find out a generator of the function field of $B/S_5$. 
Let $\rho:B\lra B$ and $\kappa:B\lra B$ be the automorphism corresponding to 
$(12345),(12)\in S_5$ respectively. Notice that $\rho,\kappa$ generate $S_5$.  
Put $T=(x_1x_2x_3x_4)/x^4_5$ and 
\begin{equation}\label{t}
t:=\sum_{i=1}^5\rho^i (T)+\sum_{i=1}^5 \kappa \rho^i (T).
\end{equation}
By direct computation, we can check $t$ is invariant under the action of $S_5$. 
\begin{thm}\label{belyi2}Keep the notation being as above. 
The rational function $t$ is a generator of $\Q(B/S_5)$. It gives rise to an algebraic Belyi function 
over $\Q$:
$$B\lra \mathbb{P}^1,\ [x_1:x_2,x_3:x_4:x_5]\mapsto t.$$ 
\end{thm}
\begin{proof}
Put $v:=x_1/x_0$ for simplicity. 
By using Grobner basis, we have $t=\ds\frac{-2g(v)}{v^4 (1 + v)^4 (1 + v^2)^4}$ where 
$$g(v)=1 + 5 v + 15 v^2 + 35 v^3 + 65 v^4 + 101 v^5 + 135 v^6 + 
   155 v^7 + 165 v^8 + 165 v^9 + 161 v^{10} $$
   $$+ 165 v^{11} +165 v^{12} + 
   155 v^{13} + 135 v^{14} + 101 v^{15} + 65 v^{16} + 35 v^{17} + 15 v^{18} + 
   5 v^{19} + v^{20}.$$
   For a generic $t\in \Q$, we have the number of $v$ satisfying the above equation is 20. 
Further, for each $v$, we have $6$ pairs of $(x_2/x_5,x_3/x_5)$ modulo the equality 
$\ds\sum_{i=1}^5 x_i=0$. 
Thus, $[\Q(B):\Q(t)]=20\times 6=120$. 
On the other hand, we observe that $[\Q(B):\Q(B/S_5)]=[\C(B):\C(B/S_5)]=|\Delta(2,3,5)/\Delta|=|S_5|=120$ and 
$\Q(t)\subset \Q(B/S_5)\subset  \Q(B)$. Thus, the claim follows. 
\end{proof}
\begin{rmk}Since any {\rm(}compact{\rm)} Shimura curve does not posses any real points {\rm(}cf. \cite{Shimura}{\rm)},  
$t$ can not generate the function field of any rational model over $E=\Q(\sqrt{5})$ of $X_{\Delta(2,4,5)}$. 
\end{rmk}

\section{A relation between two algebraic models}\label{tab}
In this section, we study the field of definition of isomorphisms between $B$, $X(\G_B)^{{\rm alg}}$ and the 
Hulek-Craig curve. 
It is known that Bring's curve $B$ is isomorphic to the algebraic model $X(\G_B)^{{\rm alg}}$ of $X(\G_B)$ 
over $\bQ$. We would like to discuss a number field $K$ such that $B$ is isomorphic over $K$ to $X(\G_B)^{{\rm alg}}$.  Obviously, $K$ cannot be $\Q$ since $B(\Q)$ is empty because of the second quadratic equation in (\ref{Bring}) 
while obviously $[1:1:1:0]$ is a $\Q$-rational point on $X(\G_B)^{{\rm alg}}$.  
For another reason, it is known that ${\rm Jac}(B)\stackrel{\Q}{\sim}
E^4$ (cf. \cite[p.87, Exercise-b)]{Se08}) where $E$ is given in Proposition \ref{decomp}.  
Thus, the field $K$ cannot be $\Q$ since ${\rm Jac}(B)$ is not isogenous over $\Q$ to 
${\rm Jac}(X)$ (see also Remark \ref{decomp-remark}). We should note that ${\rm Jac}(X)$ is isogenous over 
$L$ to ${\rm Jac}(B)$, but it is not necessarily true that $X$ is isomorphic over $L$ to $B$. 
Thus, we need more careful study for our purpose.

However, fortunately, it is known from \cite[Proposition 2.9 and its proof]{BDH} that $B$ is isomorphic over $\Q(\zeta_5)$ to Hulek-Craig's curve $X^{{\rm HC}}$ which is 
given by the desingularization of a projective geometrically irreducible singular curve
$$F(x_1,y_1,z_1):=x_1(y^5_1+z^5_1)+x^2_1y^2_1z^2_1-x^4_1y_1z_1-2y_1^3z^3_1=0$$
inside $\mathbb{P}^2$ with the coordinates $[x_1:y_1:z_1]$. This curve has studied in \cite{H} and \cite{C}.  
As explained in \cite[p.27]{C} that Hulek-Craig's curve has 
the parametrization in terms of modular forms:
$$x_1=\sum_{n\in \Z}(-q)^{(5n)^2},\ y_1=\sum_{n\in \Z}(-q)^{(5n+1)^2},\ 
z_1=\sum_{n\in \Z}(-q)^{(5n+2)^2},\ q=e^{2\pi \sqrt{-1}\tau}$$
where we slightly modified $q$ with $-q$ from Craig's notation. 
In terms of the Jacobi's theta function 
$\theta_{a,b}(z,\tau)$, $(z,\tau)\in \C\times \mathbb{H}$ with the rational character $(a,b)\in \Q^2$ (see \cite[p.10]{Mum}), we see that 
$$x_1=\theta_{0,0}(0,25(2\tau+1)),\ y_1=\theta_{\frac{1}{5},0}(0,25(2\tau+1)),\ 
z_1=\theta_{\frac{2}{5},0}(0,25(2\tau+1))$$
and thus, $x_1,y_1,z_1$ are modular forms of weight $\frac{1}{2}$ with respect to 
the double covering of $\G(5)\cap \G_1(4)$ (see \cite[Chapter IV]{Kob}).
It is known (cf. \cite[Lemma 2.6]{BDH}) that 
$$v_1=\frac{(Y^3-X)dX}{F_Y(X,Y,1)},\ 
v_2=\frac{(Y^2 X-1)dX}{F_Y(X,Y,1)},\ 
v_3=\frac{(Y-X^2)dX}{F_Y(X,Y,1)},\ 
v_4=\frac{Y(X^2-Y)dX}{F_Y(X,Y,1)}$$
where $X=x_1/z_1,\ Y=y_1/z_1$, and $F_Y(X,Y,1):=\frac{\partial F(X,Y,1)}{\partial Y}$. 
Then, in terms of $q$-expansion, we have 
$$(v_1,v_2,v_3,v_4)=
(-q^3 + q^8 + 4 q^{13}\cdots ,q^2 - 2 q^7 - q^{12}\cdots, -q^4 + 2 q^9 + 2 q^{14}\cdots, -q + q^6 + 3 q^{11}\cdots  )dq$$
$$=(f_1,f_2,f_3,f_4)A(dq),\ 
A=\left(\begin{array}{cccc}
-\frac{1}{4} & -\frac{1}{4} & -\frac{1}{4} & -\frac{1}{4}  \\   
\frac{1}{4}  & \frac{1}{4} & -\frac{1}{4} & -\frac{1}{4}  \\   
              0& 0 & \frac{1}{2} & -\frac{1}{2}  \\   
-\frac{1}{2} & \frac{1}{2} & 0   & 0
\end{array}
\right)
$$
Put $w_i:=F_Y(X,Y,1)\ds\frac{v_i}{dX}$ for $i=1,2,3,4$ and let us regard $[f_1 dq:f_2 dq:f_3 dq:f_4 dq]$ with 
the coordinates $[x:y:z:w]$ of $X$ in view of the theory of canonical curves. 
Then, $[x:y:z:w]=[w_1:w_2:w_3:w_4]A^{-1}$. Using this, we have a birational map $\psi$ from 
the singular model of $X^{{\rm HC}}$ defined by $F(x,y,z)$ to $X(\G_B)^{{\rm alg}}$ given by 
$\psi([x_1:y_1:z_1])=$
\begin{equation}\label{HC}
\psi([x_1:y_1:z_1])=[x:y:z:w]
\end{equation}
where 
$$\left\{\begin{array}{l}
x=(z_1-y_1) (x^2_1+x_1 y_1+y^2_1+x_1 z_1+z^2_1),\\
y=(y_1-z_1) (-x^2_1+x_1 y_1+y^2_1+x_1 z_1+2 y_1 z_1+z^2_1),\\
z=(y_1+z_1) (-x^2_1+y_1 z_1),\\
w=x_1 y^2_1-y^3_1+x_1 z^2_1-z^3_1.
\end{array}\right.
$$
Then, it yields an isomorphism over $\Q$ from $X^{{\rm HC}}$ to $X(\G_B)^{{\rm alg}}$. 
This result is not obvious from the construction explained in \cite[p.82-83]{H}.  
Thus, we have checked the following result:
\begin{thm}\label{isom}Recall that the projective smooth model $X(\G_B)^{{\rm alg}}$ over $\Q$ of the modular curve $X(\G_B)$ from (\ref{X}).
It holds that 
\begin{enumerate}
\item Bring's curve $B$ is isomorphic over $\Q(\zeta_5)$ to $X^{{\rm HC}}$ whose isomorphism is given by \cite[Proposition 2.9]{BDH}, and  
\item Hulek-Craig's curve $X^{{\rm HC}}$ is isomorphic over $\Q$ to $X(\G_B)^{{\rm alg}}$ whose isomorphism is given by {\rm (\ref{HC})}. In particular, Hulek-Craig's curve is isomorphic over $\Q$ to 
a canonical model of $X(\G_B)$ {\rm(}see \cite{KM} for canonical models of modular curves{\rm)}.
\end{enumerate} 
\end{thm}
Thus, $B$ inherits two algebraic Belyi functions 
such that one is defined  over $\Q(\zeta_5)$ and non-regular (not a Galois covering) while 
the other is defined over $\Q$ and regular (a Galois covering) (see \cite{G} for the 
definition of regular Belyi functions).

\end{document}